\documentclass[11pt,a4paper,reqno]{amsart}

\usepackage{graphicx,pdfsync,amssymb,mathrsfs,amsfonts,amsbsy,color,esint,mathtools,tikz,mdframed}
\usetikzlibrary{positioning}

\usepackage[shortlabels]{enumitem}

\usepackage{hyperref}
\hypersetup{
colorlinks=true,
linkcolor=blue,
filecolor=blue,
urlcolor=blue,
citecolor=blue,
}

\usepackage[utf8]{inputenc}
\usepackage[T1]{fontenc}

\newtheorem{theorem}{Theorem}[section]
\newtheorem{lemma}[theorem]{Lemma}
\newtheorem{proposition}[theorem]{Proposition}
\newtheorem{definition}[theorem]{Definition}

\setlist[enumerate]{itemsep=3pt}
\setlist[itemize]{itemsep=3pt}

\def \TT  {\mathbb{T}} 
\def \RR {\mathbb{R}}  
\def \ZZ {\mathbb{Z}}  

\def \p {\partial}
\def \g {\gamma}
\def \k {\kappa}

\def \ep {\epsilon}
\def \om {\omega}
\def \Om {\Omega}

\def \T {\mathbf{T}}
\def \N {\mathbf{N}}

\numberwithin{equation}{section}

\begin{document}

\title[Uniqueness of $\alpha$-SQG patch]{Eulerian uniqueness of the $\alpha$-SQG patch problem}

\author{Xiaoyutao Luo}

\address{Morningside Center of Mathematics, Academy of Mathematics and Systems Science, Chinese Academy of Sciences, Beijing  100190, China}

\email{xiaoyutao.luo@amss.ac.cn}


\date{\today}

\begin{abstract}
We consider the patch problem of the $\alpha$-SQG equation with $\alpha=0$ being the 2D Euler and   $\alpha= \frac{1}{2}$ the SQG equations respectively.  In the Eulerian setting, we prove  the uniqueness of patch solutions of regularity $W^{2, \frac{1}{1-2\alpha}  +} $ when $0<\alpha< \frac{1}{2}$   and  $C^{1, 4\alpha+ }$  when $0<\alpha< \frac{1}{4} $. The proof is intrinsic to the modified Biot-Savart law and   independent of the local existence of patch solutions.  
\end{abstract}

\date{\today}

\maketitle

 \section{Introduction}\label{sec:intro}
\subsection{The \texorpdfstring{$\alpha$}{a}-SQG equations}

Consider the $\alpha$-SQG equations 
\begin{equation}\label{eq:aSQG}
\begin{cases}
	\p_t \omega + (v \cdot \nabla ) \omega =0 &\\
	v(x, t) =  \nabla^\perp(-\Delta)^{-1+\alpha}\om  &
\end{cases}
\end{equation}
where the relation between the scalar $\om :\RR^2 \times [0,T] \to \RR$ and the velocity $v :\RR^2 \times [0,T] \to \RR^2$ at each time $t$ is determined by the (modified) Biot-Savart law
\begin{equation}\label{eq:aSQG_BS}
\nabla^\perp(-\Delta)^{-1+\alpha}\om  : =  c_\alpha  \int_{\RR^2} \frac{(x-y)^\perp }{|x - y|^{2+2\alpha}} \omega(y, t)\, dy,
\end{equation}
with $x^\perp = (-x_2, x_1)$ for any $x\in \RR^2$.

In this paper, for simplicity, we adjust the Biot-Savart law to have $c_\alpha= 1 ;$ this is equivalent to simple time rescaling.
The key parameter $\alpha \geq 0$ in \eqref{eq:aSQG_BS} dictates the regularity of the velocity field compared to the scalar $\om$. In our setting, $\alpha =0$ corresponds to the 2D Euler equation while $\alpha = \frac{1}{2}$ to the SQG case~\cite{MR1304437}.

In recent years there has been a surge of interest in the analysis of the $\alpha$-SQG equations, particularly its patch problem~ \cite{MR2142632,MR2397460,MR3181769,MR3466846,MR3462104,MR3549626,MR3666567,2110.08615,MR4235799,asqg_nosplash,2112.00191,MR4312192,2402.06364}. We refer readers to references  therein for more interesting aspects of the $\alpha$-SQG patch problem.

\subsection{The \texorpdfstring{$\alpha$}{a}-SQG patch problem}

The main focus of the paper is the uniqueness of the patch dynamics of the $\alpha$-SQG equations. Patch solutions are a special class of weak solutions where characteristic functions of planer domains are transported by the velocity \eqref{eq:aSQG_BS}. 

In stack contrast to the vortex patches of the Euler equations, where the boundary regularity persists globally in time~\cite{MR1235440,MR1270072,MR1207667}, the wellposedness of the patch problem of the $\alpha$-SQG equations has been challenging even the local wellposedness. The first local wellposedness result was by Rodrigo \cite{MR2142632} for the SQG case using a Nash-Moser iteration. Local wellposedness was later extended to the Sobolev setting in works~\cite{MR2397460,MR3666567,MR4235799,2105.10982} since the introduction of contour equations by Gancedo \cite{MR2397460}. See also~\cite{MR3961297,2310.15963,2311.07551,2402.06364} for related questions of long-time behaviors of the contour equations.

While global  existence results for some   classes of  patch solutions have been obtained  \cite{MR3466846,MR3462104,MR3961297,2110.08615,2303.03992}, the global regularity of the general Cauchy problem of $\alpha$-SQG patches remains unknown for all values $\alpha>0$. In a remarkable paper~\cite{MR3549626}, the authors proved finite time singularity for a related model of $\alpha$-SQG, where \eqref{eq:aSQG} is considered in an upper-half plane with patches touching the fixed boundary at all times. This model was revisited in \cite{MR4235799} and more recently \cite{2305.02427,2305.04300}. It remains an outstanding open problem whether the original patch problem  can develop a finite time singularity.

The uniqueness of the $\alpha$-SQG patch  is  of funational importantce yet one subtle  problem. Since the velocity ceases to be Lipschitz, patch solutions  are not expected to be unique within the class of localized and bounded weak solutions (see \cite{MR3987721} nonuniqueness of solutions of the SQG  with negative smoothness). Instead, one aims for uniqueness with the class of patch solutions.

Another caveat in dealing with the uniqueness issue of the patch solutions is whether   the uniqueness of solutions constructed by the contour equations \cite{MR2397460} implies that of the patch problem~\cite{MR3666567,MR3748149}.  In~\cite{MR2397460}, Gancedo modified the tangential velocity and introduced the contour equation that satisfies favorable energy estimates. These estimates in turn give the local existence and uniqueness of the contour equations. In~\cite{MR3666567,MR4235799}, the authors are able to upgrade such uniqueness of contour equations to the patch problem.

We summarize the state-of-the-art uniqueness results for $\alpha$-patches in the next table. A comparison with our results follows shortly.

\begin{table}[h]\renewcommand{\arraystretch}{1.30}
\begin{tabular}{|l|l|l|l|l|}
	\hline
	Uniqueness   &      Regularity & Domain & Reference \\  \hline
	Contour equation & $H^3$   ,$0< \alpha < \frac{1}{2}$    &   $\RR^2 $     &     \cite[Theorem 4.1]{MR2397460}      \\   \hline
	Patch problem    &    $H^3$   ,$0< \alpha < \frac{1}{2}$   &   $\RR^2$ and $\RR\times \RR^+$    &    \cite[Theorem 1.5]{MR3666567}       \\ \hline
	Patch problem    &  $C^{2+}$ ,$\alpha = \frac{1}{2}$     &    $\RR^2$    &    \cite[Theorem 1.2]{MR3748149}       \\ \hline
	Patch problem    & $H^2$  ,$1<\alpha < \frac{1}{6} $     &    $\RR\times \RR^+$    &    \cite[Theorem 1]{MR4235799}       \\ \hline
	Contour equation &  $H^{2+}$  ,$\alpha = \frac{1}{2}$      &    $\RR^2$    &    \cite[Theorem 1.2]{2105.10982}       \\ \hline
\end{tabular}
\label{tab}
\end{table}

\subsection{Main results}
In this paper, we denote $\Omega_t$ for a time-dependent  bounded domain (connected open set) $\Omega \subset \RR^2$  whose boundary is a simple closed curve on some interval $ t\in [0,T]$.

For Sobolev spaces $X = W^{k,p}$ and H\"older spaces $X = C^{k,\beta}$ and an arc-length parameterization $\g $ of $\p \Omega$,   we define
\begin{equation}\label{eq:intro_domain_norm}
\| \Om \|_{X} : = \| \g \|_{ X( [-L/2,L/2] )}  \quad \text{ for $X \in \{ W^{k,p}, C^{k,\beta} \}$ }
\end{equation}
where $ L $ is the length of the boundary $\p \Om$. 

We now introduce the Eulerian notion of a patch solution.
\begin{definition}\label{def:asqg_patch}
Let $\Omega_t \subset \RR^2$ be a time-dependent bounded domain  on $ t\in [0,T]$. We say $\Omega_t$ is a patch solution of $\alpha$-SQG (or an $\alpha$-patch) on $[0,T]$ if the characteristic function $\omega = \chi_{\Omega_t}$ is a weak solution to \eqref{eq:aSQG}--\eqref{eq:aSQG_BS}, i.e. for any $ \varphi \in C^\infty_c(\RR^2 \times [ 0,T ) )$,
\begin{equation}\label{eq:def:asqg_patch}
	\int_{\Omega_0} \varphi(\cdot,0)  \, dx = -\int_{0}^T \int_{\Omega_t}     \p_t \varphi + v \cdot \nabla \varphi   \,dx dt,
\end{equation}
where $v = K_\alpha * \omega$ is given by the Biot-Savart law \eqref{eq:aSQG_BS}.

In addition, for a regularity class $X \in \{ W^{k,p}, C^{k,\beta} \}$ we say $\Omega_t$ is a $X$   patch solution  if  $\sup_{t \in [0,T]} \| \Om_t \|_{X }<\infty$.
\end{definition}

The above Eulerian definition encompasses the previous notions of patch solutions, cf. \cite[p.4]{MR3748149} and \cite[Definition 1.2]{MR3666567}. Since the velocity field is not Lipschitz, it is not clear \emph{a priori} that a notion of a Lagrangian flow exists, and in fact it is only proved very recently in our joint work \cite{asqg} with A. Kiselev for $W^{2,\frac{1}{1 -2\alpha } +}$ patch solutions,  matching the uniqueness threshold obtained here.

We now state the main theorems of the paper.

	\begin{theorem}[$W^{2, \frac{1}{1 -2\alpha } + }$ Uniqueness]\label{thm:main_Sobo}
		Let $0<\alpha < \frac{1}{2}$.  If $  p > \frac{1}{1 -2\alpha } $, then $W^{2, p }$ patch solution of  $\alpha$-SQG (in the sense of Definition \ref{def:asqg_patch}) are unique.
		
		Precisely, if $\Omega_t $ and $\widetilde \Om_t$  are  $W^{2, p }$ two $\alpha$-patches on $[0,T]$ with the same initial data $\Om_0$ (in the sense of Definition \ref{def:asqg_patch}),  then $\Omega_t $ and $\widetilde \Om_t$ must coincide.
	\end{theorem}

By comparison, we lower the regularity from $H^3   $ of \cite{MR3666567}   to   $W^{2, \frac{1}{1-2\alpha} + }$ and also extend the uniqueness statement to the Eulerian setting. As a corollary, the   $H^2$-patch solutions constructed by  \cite[Theorem 3]{MR4235799} are unique in the Eulerian setting when   $0< \alpha < \frac{1}{4}$.

Our second uniqueness result concerns patch solutions with rougher boundaries for smaller values of $\alpha>0$. In fact,  Theorem \ref{thm:main_Holder} below   provides the uniqueness of the patch problem  with less than two total derivatives for the first time.

	\begin{theorem}[$C^{1,4\alpha + }$ Uniqueness]\label{thm:main_Holder}
		Let $0<\alpha < \frac{1}{4}$. If $  \beta >4\alpha $, then $C^{1, \beta }$ patch solution of  $\alpha$-SQG (in the sense of Definition \ref{def:asqg_patch}) are unique.
	\end{theorem}

Note that  in  \cite{MR3748149}, the authors prove uniqueness of patches of SQG  (i.e. $\alpha = \frac{1}{2}$) via an ingenious reparameterization process, for any contour $\g \in C_t C^{2,0+} \cap C^1_t   C^{1}$. Even though our main results (i.e. $\alpha < \frac{1}{2}$) are not directly comparable to that of \cite{MR3748149},  in the setting of Definition \ref{def:asqg_patch} we do not require this kind of smoothness--- it is not clear a priori that  such parameterizations exist in the setting of Definition \ref{def:asqg_patch}.

\subsection{Ideas}

\usetikzlibrary{shapes.geometric, arrows}

\tikzstyle{rec} = [rectangle, rounded corners, 
minimum width=3cm, 
minimum height=0.8cm,
text centered, 
draw=black]
 
\tikzstyle{arrow} = [thick,->,>=stealth]
 
\begin{figure}[h]
\begin{tikzpicture}[node distance=7cm]

\node (node1) [rec] {Eulerain patches $\Omega_{i,t}$};
\node (node2) [rec, right of=node1] {Lagrangian flows $\p_t \g_i = v_i(\g_i ) $ of $\p \Omega_{i,t}$};
\node (node3) [rec, below of=node1, yshift=5.5cm] {Stability of $ \T_1 - \T_2  $ and $ g_1 - g_2 $};
\node (node4) [rec, right of=node3] {Dynamics of $g_i  = |\g_i | $   and $\T_i = |g_i |^{-1} \p_x \g_i $};

\draw [arrow] (node1) -- (node2);
\draw [arrow] (node2) -- (node4);
\draw [arrow] (node4) -- (node3);

\end{tikzpicture}
\end{figure}

Unlike the previous works~\cite{MR2397460,MR3666567,MR4235799,2105.10982}, we do use Gancedo's contour equations. Instead, we follow the Lagrangian framework developed in \cite{asqg}.  The proof of both theorems relies on the stability estimates of two patch solutions in terms of their tangent vector and arc-length metric.

The first step of the process is to deduce the wellposedness of the flow from the Eulerian setting of Definition \ref{def:asqg_patch} as in \cite{asqg}. Then this Lagrangian flow allows us to study the dynamics of the patch boundary, tracking the evolution of tangent vector and arc length. The stability estimates in turn rely on the dynamics of the trajectories of the patch boundaries.

One key ingredient is  certain odd symmetry in the evolution of both the tangent vector and the arc length, similar to the curvature equations used in \cite{asqg}. Once the patch solutions have suitable regularity, the odd symmetry can be used to deduce $L^2$-stablity of the tangent vector and arc-length metric. We refer to Section \ref{sec:stablity} for more details.

\subsection{Plan of the paper}

The rest of the paper is organized as follows.
\begin{itemize}

\item After necessary preliminaries in Section \ref{sec:prelim}, we introduce and then prove the wellposedness of the flow of patch solutions in Section \ref{sec:flow}.

\item The main uniqueness theorems are proved in   Section \ref{sec:proof} assuming the key stability estimate, Theorem \ref{thm:uniqueness_by_flow}, whose proof is the   content of Section \ref{sec:stablity}.

\end{itemize}

\subsection*{Acknowledgments}

The author is grateful to A. Kiselev for stimulating discussions.

\section{Preliminaries}\label{sec:prelim}

\subsection{Notations}

Functions on the torus $\TT = \RR / 2\pi  \ZZ $ are identified with $2\pi$-periodic functions on $\RR$. Similarly, if $f$ is $L$-periodic for some $L>0$, we consider it a function on the rescaled torus $\TT = \RR / L\ZZ $.

We use $|f|_{L^{ p} }$,$|f|_{W^{k,p} }$ and $|f|_{C^{k,\beta}}$ to denote various   norms  over various intervals, such as $[-L/2, L/2 ]$ for the arc-length parameterizations,  without spelling out the specific spatial domain.

Throughout the paper, $X \lesssim Y$ means $X \leq CY$ for some constant $C>0$ that may change from line to line. Similarly, $ X \gtrsim Y$ means $X \geq C Y$ and $X \sim Y$ means $ X \lesssim Y$, and $X \gtrsim Y$ at the same time.

We also use the following big-O notation: $X = O(Y)$ for a quantity $X$ such that $|X| \leq C Y$ for some absolute constant $C>0$ and object $Y>0$.

\subsection{Planar curves and domains}

The boundary of a  $C^1$ patch (a simply-connected bounded domain) is a  $C^1$ curve, the image of a $C^1$ periodic map  $\gamma : \RR \to \RR^2$ such that $ |\g'|>0$. Here $\g$ is also called  a parametrization  and if $ |\g'| = 1 $, we say it is an arc-length parametrization. We always assume the parameterization is counterclockwise oriented.

For a Banach space $X \subset C^1 $, we say a curve is of class $X$ if its arc-length parameterizations are of class $X$. Given a curve, we write $\T$ as the unit tangent vector and $\N = - \T^\perp$ the outer unit normal vector. The letter $g$ is reserved for the metric $g = |\p_x \g |$ of a parameterization.  

Throughout the paper, we consider curves and domains that are time-dependent.  In this case, $\p \Omega_t$ refers to the boundary   of $\Omega_t$ at each time $t$.

\subsection{Rough parametrizations of  regular domains}\label{subsec:rough_gamma_regular_Omega}

In what follows, we often consider parameterizations that are less regular than the curve itself.  For example when $t\mapsto \theta(t)$ is only $ C^{1+}$ and monetone, the map $t \mapsto (\sin(\theta(t)) ,  \cos(\theta(t)) )$ provides a $ C^{1+}$ parameterization of the unit circle in $\RR^2$.

For example, let $\Omega$ be a $C^{1,\beta }$ bounded domain and let $\g \in C^{1,\sigma}(\TT)$ be a parametrization of $\p \Omega$ for some $\sigma>0$. In this case, the tangent $\T(x)$ and normal $\N(x)$ in the parameterization $\g$ (namely defined by the corresponding objects at point $\g(x) $ for any $x \in \TT$) satisfy  $\T, \N \in C^{0,\beta }(\TT)$, the expected regularity.

\subsection{Estimates for \texorpdfstring{$C^{1, \beta  }$}{C1b}  and \texorpdfstring{$W^{2, p  }$}{W2p} curves}\label{subsec:curves_estimates}

By definition, in the arc-length parameterization, a $C^{1 ,\beta}$ curve $\g$ satisfies the usual bounds
\begin{align*}
 \T(s) \cdot \N(s')  & = O(|s-s'|^\beta) \\
 \T(s) \cdot \T(s')  & = 1+ O(|s-s'|^\beta)  . 
\end{align*}

However, when the curve is $W^{2,p}$, we get improved bounds on certain quantities due to the $L^p$-integrable curvature, such as \eqref{eq:arc_length_a}, \eqref{eq:arc_length_c}, \eqref{eq:arc_length_d}, and \eqref{eq:arc_length_e}  below.

Given any $L$-periodic function $f:\RR \to \RR   $,  we denote by $\mathcal{M}f: \RR \to \RR$ the ($L$-periodic) maximal function of $f$,
\begin{equation}\label{eq:def_maximal_function}
\mathcal{M}f(x ) = \sup_{0< \ep < 2 L }\frac{1}{2\ep}\int_{ x-\ep}^{x + \ep} |f ( y )| \, d y .
\end{equation}
The restriction $\ep < 2 L $ is non-essential and the boundedness of periodic $\mathcal{M}$  on $L^p$ for $1<p\leq \infty$ is also standard  \cite{MR0290095}.

\begin{lemma}\label{lemma:arc_length_estimates}
Let  $\Om $ be a $W^{2,p}$ domain for some $1< p \leq \infty$ and set $\beta = 1 -\frac{1}{p}$. Let $\g $ be an arc-length parameterization of $ \p\Om$. For  any $s,  s' \in \RR$, we have
\begin{subequations}
\begin{align} 
\T(s)  \cdot \T( s' ) &= 1+ O(  |s - s'  |^{2\beta }) \label{eq:arc_length_a}\\
(\g(s) - \g( s' )) \cdot \N(s) &= O(  |s - s'  |^{1+\beta })  \label{eq:arc_length_b} \\
(\T(s) - \T( s' )) \cdot \T(s) &= O(  |s - s'  |^{2\beta }) \label{eq:arc_length_c} \\
\frac{ |s -s'  | }{ |\g(s) - \g( s' )| }  &= 1+ O(  |s - s' |^{2\beta }) \label{eq:arc_length_d} \\
(\g(s) - \g( s' )) \cdot \T(s) &= (s - s'  )+ O(  |s - s'  |^{1+ 2 \beta }) \label{eq:arc_length_e}
\end{align}
\end{subequations}
and the maximal estimates
\begin{subequations}
\begin{align}
\T( s' ) \cdot \N(s) & =  O(  \mathcal{M}\k ( s ) |s - s'  | ) \label{eq:arc_length_M_a} \\    
(\g(s) - \g( s' )) \cdot \N(s) &= O(  \mathcal{M}\k ( s ) |s - s'  |^{2})  \label{eq:arc_length_M_b} \\
  \T(s)  \cdot \left[\T(s) -\T( s' )  \right] &= O(  \mathcal{M}\k ( s ) |s - s'  |^{1+\beta}) \label{eq:arc_length_M_c}\\
\left[(\g(s) - \g( s' ) \right]\cdot \left[\T(s) -\T( s' )  \right] &= O(   \mathcal{M}\k (s ) |s - s'  |^{2+\beta}) \label{eq:arc_length_M_d}.
\end{align}

\end{subequations}

\end{lemma}
\begin{proof}
In the first group of estimates, \eqref{eq:arc_length_a},\eqref{eq:arc_length_b},\eqref{eq:arc_length_c}, and \eqref{eq:arc_length_e} follow from the basics bound $ |\T (s)\cdot \N (s') | \lesssim |s -s'|^{\beta }$ and the fundamental theorem of calculus. With these bounds, a Taylor expansion argument applied to $\frac{ |s -s'  |^2 }{ |\g(s) - \g( s' )|^2 }$ then implies \eqref{eq:arc_length_d}.

The second group of estimates \eqref{eq:arc_length_M_a}--\eqref{eq:arc_length_M_d} are a direct consequence from the definition \eqref{eq:def_maximal_function}.

\end{proof}
\subsection{The Biot-Savart law for \texorpdfstring{$\alpha$}{a}-patches}

Let  $\Om$ be a $C^{1}$ domain and $\g$ be a $C^1(\TT)$ (counter-clockwise) parameterization of  $\p \Om$. The velocity field $v$ on $\p \Omega$ given by the Biot-Savart \eqref{eq:aSQG_BS} satisfies
\begin{equation}\label{eq:aux_BS_gamma}
v (x) =    -\frac{1}{2\alpha } \int_{\TT} \frac{\T(y)}{|\g(x) - \g(y)|^{2\alpha }} g (y)\, dy ,
\end{equation}
where $ g = |\p_x \g |$ is the arc-length of $\g$.

If the parameterization is arc-length, we instead use the labels $s,s'$ and
\begin{equation}\label{eq:aux_BS_arc-length_gamma}
v (s) =     -\frac{1}{2\alpha } \int_{\g} \frac{\T(s')}{|\g(s) - \g(s')|^{2\alpha }}  \, ds' .
\end{equation}
Here, the velocity $ v$ is   considered a function on the patch boundary $\p \Om$, a convention that we will follow often in what follows.

\section{Wellposedness of the flow of \texorpdfstring{$C^{1,\beta}$}{C1b}  patches }\label{sec:flow}

The general notion of   patch solutions  used in Definition \ref{def:asqg_patch} provides little information on how the patches are evolving. To prove their uniqueness, we first pass to the Lagrangian setting and use the particle trajectories to study  the dynamics of the patch boundary. 

\subsection{The flow of \texorpdfstring{$\alpha$}{a}-patches}\label{subsec:flow_asqg_patches} 

Among all the possible parameterizations of a patch boundary, there is a ``canonical'' one, evolving according to the Biot-Savart law \eqref{eq:aSQG_BS}. This is   introduced in \cite{asqg} as the flow of a patch solution.

\begin{definition}
Let $ 0< \alpha < \frac{1}{2}$ and $\Omega_t$ be an $\alpha$-patch on $[0,T]$. We say $\g:\TT\times [0,T] \to \RR^2$ is a flow of the patch $\Omega_t$ if
\begin{itemize}
	\item The map $ x  \mapsto \g( x ,t)$ is a  parameterization (i.e. $|\g'|>0$) of $\p \Om_t$ for every $t\in [0,T]$;
	\item The flow equation $\p_t \g = v(\g,t )$ holds on $\TT \times [0,T]$, where $v(x,t) :\RR^2 \times [0,T] \to \RR^2$ is the velocity field generated by the patch $\Omega_t$.
\end{itemize}

If in addition for some $\sigma>0$, $\g(\cdot, t) $ is of class $C^{1,\sigma}$ for all $t\in [0,T]$ (namely $ \g \in L^\infty( [0,T] ;C^{1,\sigma}(\TT) )$), we say $\g$ is a $C^{1,\sigma}$ flow of $\Omega_t$.
\end{definition}

The existence of a flow with the same regularity as that of the patch solution is generally not obvious. In general, one expects a loss of regularity of the corresponding flow compacted to the patch solution itself, as demonstrated in  \cite{asqg}.

We will use the main result Theorem 3.11 from \cite[Section 3]{asqg}:
\begin{theorem}[{\cite[Theorem 3.11]{asqg}}]\label{thm:flow_existence}
Let $0< \alpha< \frac{1}{2}$ and    $\beta >  2\alpha $. Suppose  $\Omega_t$ is a $C^{1,\beta}$ $\alpha$-patch with the initial data $\Omega_0$ and $v =K_\alpha  *  \chi_{\Omega_t}:\RR^2\times [0,T] \to \RR^2$ is the velocity field of $\Omega_t$.

For any initial parametrization of $ \p \Omega_0$, $\g_0\in C^{1,\beta}(\TT) $,  the Cauchy problem for the flow $\gamma: \TT \times [0,T] \to \RR^2$
\begin{equation}
	\begin{cases}
		\p_t \gamma(x,t) = v(\gamma(x,t) , t) &\\
		\gamma |_{t=0} = \gamma_0&
	\end{cases}
\end{equation}
has a unique solution $\gamma \in  C([0,T]; C^{1,\sigma }(\TT))$ for  $\sigma = \beta -2\alpha >0 $.
\end{theorem}

Note that \cite[Theorem 3.11]{asqg} only proved the existence and uniqueness of the flow for $W^{2,p}$ patch solutions $p > \frac{1}{1 - 2\alpha}$. The proof hinges on several estimates of the velocity field, where   only \cite[Lemma 3.6]{asqg} was stated for $W^{2,\frac{1}{1 - 2\alpha}+}$ patch solutions and the rest  of the proof holds for $ C^{1, 2\alpha + }$ patch solutions. Therefore, to prove Theorem \ref{thm:flow_existence}, we just need to supply the below Lemma \ref{lemma:estimate_dv} to replace \cite[Lemma 3.6]{asqg}.

\begin{lemma}\label{lemma:estimate_dv}
Let $0< \alpha< \frac{1}{2}$ and $ \beta  >  2\alpha    $. Suppose  $\Omega $ is a $C^{1, \beta }$ bounded domain and $v$ is the velocity field of $\Omega $ according to \eqref{eq:aSQG_BS}.

For any arc-length parameterization of $\p \Omega $, the velocity field satisfies
\begin{align}\label{eq:dv_C1beta}
	\p_s v (s)  & =     P.V.  \int_{ \g    } \T (s' )   \frac{   (\g(s  )-\g( s' )) \cdot \T(s ) }{  |\g(s )-\g(s' )|^{2  +2\alpha }}  \,ds' ,
\end{align}
and the estimate
\begin{equation}\label{eq:estimate_v_C1beta}
	\big|   \p_s v    \big|_{C^{1,\sigma }} \leq C(\alpha,p , \Omega)
\end{equation}
where $\sigma = \beta -2\alpha >0 $.
\end{lemma}
\begin{proof}
To prove the identity \eqref{eq:dv_C1beta}, one notices first that the right-hand side integral is uniformly bounded on $ \g$ thanks to the $C^{1,\beta}$ regularity of $\g$. Then \eqref{eq:dv_C1beta} follows from \eqref{eq:aux_BS_arc-length_gamma} via a distributional argument.

The proof of \eqref{eq:dv_C1beta} follows nearly verbatim    \cite[Lemma 3.6]{asqg}. 
We start with a reparametrization of the integral
\begin{align}\label{eq:aux_c1b_dsv_1}
	\p_s v (s)  & =     P.V.  \int_{ \g    } \T (s'+s )   \frac{   (\g(s  )-\g( s'+s )) \cdot \T(s ) }{  |\g(s )-\g(s'+s )|^{2  +2\alpha }}  \,ds' .
\end{align}
Denote $\Delta_\delta f = f( s+\delta ) - f(s)$ for a function $f$ and $\delta>0$. It suffices to show that
\begin{align}\label{eq:aux_c1b_dsv_2}
\Big|\Delta_\delta [ 	\p_s v  ]   \Big|  & \lesssim \delta^{ \beta - 2\alpha}
\end{align}
using the split $\Delta_\delta [ 	\p_s v  ] = I_1 + I_2 $
where
\begin{align*}
I_1 & = \lim_{\ep \to 0^+} \Big(   \int_{\ep \leq |s'| \leq 2\delta  } \int  \T (s'+s +\delta)    \frac{   (\g(s +\delta )-\g( s'+s +\delta)) \cdot \T(s +\delta) }{  |\g(s+\delta )-\g(s'+s +\delta)|^{2  +2\alpha }}  \,ds'   \\
 &  \qquad \qquad - \int_{\ep \leq |s'| \leq 2\delta  } \int  \T (s'+s  )    \frac{   (\g(s   )-\g( s'+s  )) \cdot \T(s  ) }{  |\g(s  )-\g(s'+s  )|^{2  +2\alpha }}  \,ds'   \Big),
\end{align*}
and
\begin{align*}
I_2 & =      \int_{2\delta \leq |s'|    }   \Delta_\delta \ \left[    \T (s'+s  )    \frac{   (\g(s   )-\g( s'+s  )) \cdot \T(s  ) }{  |\g(s  )-\g(s'+s  )|^{2  +2\alpha }}  \right]  \,ds'  .
\end{align*}

\noindent
\textbf{\underline{Estimate of $I_1$:}}

By the    $C^{1,\beta}$ regularity of $\g$ again, for any $s,s \in \RR$ there holds
\begin{equation}\label{eq:aux_c1b_dsv_3}
\begin{aligned}
 \T(s' ) &= \T (s)   + O(| s  - s' |^{  \beta  } ) \\
 (\g(s      )-\g(  s' )) \cdot \T(s ) &= (s  - s' ) + O(|s  - s'|^{ 1 + \beta  } ) \\
\frac{|s  - s' |^{2  +2\alpha } }{|\g(s)-\g(  s')|^{2  +2\alpha }}   &  =1 + O(|s  - s'|^{   \beta }  ) .
\end{aligned}
\end{equation}

Here we note the constants in the big-O terms are independent of $s,s'$.
Using \eqref{eq:aux_c1b_dsv_3}, we have
\begin{align*}
I_1 = \lim_{\ep\to 0+} &\int_{ \ep< |  s' |    \leq 2\delta } \T ( s+  \delta  )    \frac{ -   s'   }{  |  s'  |^{2  +2\alpha }} \,ds'  + \int_{ \ep< |  s' |    \leq 2\delta }  O(| s'|^{-1 + \beta - 2\alpha })  \,d   s' \\
&\quad -\int_{ \ep< |  s' | \leq 2\delta   }\T ( s )   \frac{   - s'   }{  |  s'  |^{2  +2\alpha }}\,ds'   + \int_{ \ep< |  s' |    \leq 2\delta }  O(| s'|^{  -1 + \beta - 2\alpha })  \,ds'    .
\end{align*}

The first terms on the right-hand side of the above two lines vanish by the oddness of the kernel $\frac{   - s'   }{  |  s'  |^{2  +2\alpha }}$, and we have
\begin{align*}
|I_1| & \lesssim     \int_{  | s'| \leq 2\delta   }   |    s' |^{   -1 + \beta - 2\alpha }      \,ds'.
\end{align*}
Since $ \beta >    2\alpha  $, the exponent $  -1 + \beta - 2\alpha > -1 $, and integrating we obtain
\begin{align*}
I_1 &= O(\delta^{  \beta -  2\alpha  }   ).
\end{align*}

\noindent
\textbf{\underline{Estimate of $I_2$:}}

In this regime, we first consider the finite difference (in variable $s$) of the integrand
\begin{align}\label{eq:aux_c1b_dsv_4}
 \Delta_\delta \left[ \T (s'+s )   \frac{   (\g(s  )-\g( s'+s )) \cdot \T(s ) }{  |\g(s )-\g(s'+s )|^{2  +2\alpha }}    \right] .
\end{align}

By telescoping and the $C^{1,\beta}$ regularity of $\gamma$ used before, we can obtain 
\begin{align}\label{eq:aux_c1b_dsv_5}
\left| \Delta_\delta \left[ \T (s'+s )   \frac{   (\g(s  )-\g( s'+s )) \cdot \T(s ) }{  |\g(s )-\g(s'+s )|^{2  +2\alpha }}    \right]   \right| \lesssim \delta^{\beta } |s'|^{-1-2\alpha}   
\end{align}
which follows from the below three estimates (recall $|s'| \geq 2\delta$):
\begin{subequations}
	\begin{align}
 \left| \Delta_\delta \left[ \T (s'+s )        \right]   \right| & \lesssim \delta^{\beta }\label{eq:aux_I_2_a} \\
\left| \Delta_\delta \left[   (\g(s  )-\g( s'+s )) \cdot \T(s )      \right]   \right| & \lesssim \delta^{\beta } |s'|  \label{eq:aux_I_2_b}  \\
\left| \Delta_\delta \left[     \frac{  1}{  |\g(s )-\g(s'+s )|^{2  +2\alpha }}    \right]   \right| & \lesssim \delta^{\beta } |s'|^{-2-2\alpha}  .\label{eq:aux_I_2_c}
\end{align}
\end{subequations}
Indeed, the first \eqref{eq:aux_I_2_a} follows from the $C^{\beta}$ regularity of $\T$; the second   \eqref{eq:aux_I_2_b}  is a consequence of the identity
\begin{align*}
 (\g(s  )-\g( s'+s )) \cdot \T(s )     = -s' - \int_{s+s'}^{s} (\T(\tau) - \T(s)) \cdot \T(s) \, d \tau ;
\end{align*}
and the third \eqref{eq:aux_I_2_c} follows from
\begin{equation} 
\begin{aligned}
 |\g(s )-\g(s'+s )|^{2   } &= \left( \T(s) s' +  \int_s^{s' +s }  (\T(\tau ) - \T(s) ) \,d \tau \right)^2  \\
& = |s'|^2 + 2 s' \int_s^{s' +s } \T(s) \cdot (\T(\tau ) - \T(s) ) \,d \tau \\
&  \qquad+  |s'|^2\Big(  \int_s^{s' +s }   \T(\tau ) - \T(s)   \,d \tau \Big)^2,
\end{aligned}
\end{equation}
and the mean value theorem for $ f(r) =  \frac{1}{r^{1+\alpha}}$ with $ r(s) = |\g(s )-\g(s'+s )|^{2   } $.

Once \eqref{eq:aux_c1b_dsv_5} is established, the bound for $I_2$ follows:
\begin{equation}
| I_2| \lesssim \int_{2\delta \leq |s'|    }   \delta^{\beta } |s'|^{-1-2\alpha}     \,ds' \lesssim \delta^{  \beta -  2\alpha  }   . 
\end{equation}

\end{proof}

\subsection{Velocity estimates along the flow}

The formula of $\p_s v$ given by Lemma \ref{lemma:estimate_dv} is stated in terms of an arc-length parameterization. To derive stability estimates of the flows, we need to rewrite it in terms of the Lagrangian label.

\begin{lemma}\label{lemma:estimate_v_W2p_Lagran}
Let $\beta  >2\alpha$ and $\Omega$ be a $C^{1,\beta}$ bounded domain. Suppose $\g \in C^{1,\sigma}(\TT)$ is a parameterization of  $\p \Omega$ for some $ \sigma >2\alpha$. 

Then the velocity field of $\Omega$ according to \eqref{eq:aSQG_BS} under the parameterization $\g$ satisfies 
\begin{equation}\label{eq:estimate_v_W2p_Lagran}
\p_s v (x)  = P.V. \int_{\TT} 	\T(x )   \frac{(\g (x) - \g (y) )\cdot \T (x)}{|\g (x) - \g (y)|^{2+2\alpha }}  g (y) \,  dy  \quad \text{for any $\g(x) \in  \p \Omega$}.
\end{equation}
 
\end{lemma}
\begin{proof}
Given a parameterization $\g \in C^{1,\sigma}(\TT)$    of  $\p \Omega$,   consider the  arc-length variable $s : = \ell( x ) = \int_0^{x} g( y )\, d y $. By Lemma \ref{lemma:estimate_dvT_W2p}, and a change of variable
\begin{equation}
\begin{aligned}
\p_s v (x) &  = \lim_{\ep \to 0 } \int_{|\ell( x ) - \ell( y )  | \geq \ep } 	\T(x )   \frac{(\g (x) - \g (y) )\cdot \T (x)}{|\g (x) - \g (y)|^{2+2\alpha }}  g (y) \,  dy \\
&  := \lim_{\ep \to 0 } I_\ep(x) .
\end{aligned}
\end{equation}
 
Let us   consider the variant $ \widetilde{I}_{\ep }(x)$
\begin{equation}\label{eq:aux_curvature_kernal_I_1_3}
\begin{aligned}
\widetilde{I}_{\ep }(x):   & =  \int_{ g(x)|    x   - y  | \geq \ep  } \T(x )   \frac{(\g (x) - \g (y) )\cdot \T (x)}{|\g (x) - \g (y)|^{2+2\alpha }}  g (y) \,  dy .
\end{aligned}
\end{equation}
We claim $\lim_{\ep \to 0^+} \widetilde{I}_{\ep }(x)= \lim_{\ep \to 0^+}  {I}_{\ep }(x) $. Consider the set $S_\ep(y)$ of the symmetric difference
\begin{equation}
 S_\ep(x): =  \{y \in \TT : | \ell(x) -\ell(y) | \geq \ep \} \Delta \{y \in \TT : g(x) |   x  - y  | \geq \ep \}
\end{equation}
and for all sufficiently small $\ep>0$, we have $| x-y | \sim \ep  $ when $ y \in  S_\ep(x)$. 

It follows that
\begin{align*}
\left|  \widetilde{I}_{11}( \ep) -  {I}_{11}( \ep)  \right| & \lesssim   \int_{S_\ep(x) }    |x-y|^{-1-2\alpha }  \, dx  \, d y\\
& \lesssim \ep^{ -2\alpha }  |S_\ep| ,
\end{align*}
where $|S_\ep|$ is the Lebesgue measure of $ S_\ep$.

Since $x \mapsto \ell  (x)$ is a $C^{1,\sigma}$ diffeomorphism, the fundamental theorem of calculus implies that for each $x\in \TT$, 
has Lebesgue measure $|S_\ep| \lesssim \ep^{\sigma} $ and hence
\begin{equation}
\left|  \widetilde{I}_{11}( \ep) -  {I}_{11}( \ep)  \right| \lesssim    \ep^{\sigma-2\alpha }  \to 0 \quad \text{as $\ep \to 0$.}
\end{equation}

\end{proof}

\subsection{Evolution of tangent and arc-length}

Given a $C^{1,\sigma}$  flow $\g $ of an  $\alpha$-patch, we now derive evolution equations for the arc-length metric $g= |\p_x \g| :\TT \times [0,T] \to \RR^+ $ and the  tangent vector $\T: \TT \times [0,T] \to \mathbb{S}^1$. 

By the integral form of the flow equation
\begin{equation}\label{eq:aux_int_flow}
\g(x,t) = \g_0 (x) + \int_0^t v( \g(x,t') , t') \,dt'
\end{equation}
and the velocity field regularity $ \p_s v \in C^{ \sigma}(\TT)$, we can differentiate \eqref{eq:aux_int_flow} to find that $ \p_x \g \in C([0,T];C^{ \sigma}(\TT)) \cap C(\TT;C^{1}([0,T] ) ) $ satisfying the equation
\begin{equation}\label{eq:aux_eq_d_gamma}
\p_t (\p_x \g) = \p_s v g.
\end{equation}

Since $ g \T = \p_x \g $,  from \eqref{eq:aux_eq_d_gamma} it follows that the metric $g \in C([0,T];C^{ \sigma}(\TT))$ satisfies the equation
\begin{equation}\label{eq:metric_evolution}
\begin{cases}
	\p_t g = g \p_s v \cdot  \T &\\
	g |_{t = 0} = |\p_x \g_0 |. &
\end{cases} 
\end{equation}
By   $ g \T = \p_x \g $, \eqref{eq:aux_eq_d_gamma}, and \eqref{eq:metric_evolution} again, we have the evolution for the unit tangent vector $\T \in C([0,T];C^{ \sigma}(\TT))$,
\begin{equation}\label{eq:tangent_evolution}
\begin{cases}
	\p_t \T = \p_s v \cdot \N \N &\\
	\T |_{t = 0} = \p_s \g_0. &
\end{cases} 
\end{equation}

We remark that these equations \eqref{eq:aux_eq_d_gamma}, \eqref{eq:metric_evolution}, and \eqref{eq:tangent_evolution} are satisfied classically and will be used in Section \ref{sec:stablity} to prove the main stability estimates of the paper.

%

\section{Proof of   uniqueness theorems}\label{sec:proof}

In this section, we conclude the proof of Theorem \ref{thm:main_Sobo} and Theorem \ref{thm:main_Holder} assuming the following result. The proof of Theorem \ref{thm:uniqueness_by_flow} is the main subject of this paper and is postponed to the next section.

\begin{theorem}\label{thm:uniqueness_by_flow}
 Let $ 0 < \alpha < \frac{1}{2}$ and $ \beta  > 2\alpha $. Suppose that  $\g_i : \TT \times [0,T] \to \RR^2$, $i=1,2$ are two $ C^{1,\beta }$ flows of two corresponding $ C^{1,\beta }$ patch solutions  $ \Omega_{i,t}$. 
 
 Let $\T_i$ and $g_i$ be the tangent and arc-length of each $\g_i$ and define
\begin{equation}
\delta(t) = |\g_1 (t) -\g_2(t) |_{L^2}^2 +  |g_1 (t) -g_2(t) |_{L^2}^2 +  |\T_1 (t)  - \T_2 (t) |_{L^2}^2  . 
\end{equation} 

Then $t\mapsto \delta(t)$ is   continuous and for all $ t\in [0,T]$
\begin{equation}
\delta(t) \leq \delta(0) + C \int_0^t \delta(\tau ) \,  d\tau.
\end{equation}
 
\end{theorem}

We emphasize that in Theorem \ref{thm:uniqueness_by_flow}, the $ C^{1,\beta }$ regularity of the flows is an important assumption and in general can not be deduced from the $ C^{1,\beta }$ regularity of patch solutions.

\subsection{Proof of \texorpdfstring{$C^{1,  4\alpha+   }$}{C1,4a+} uniqueness}
Let $ \beta > 4\alpha$. Given two  two $C^{1,\beta  }$ patch solutions $\Omega_{i,t}$ on some $[0,T]$ that coincide at $t=0$, we need to show $  \Omega_{1,t} = \Omega_{2,t}$ for all time.

First, since the initial data  $  \Omega_{ 0} =  \Omega_{1,t} = \Omega_{2,t}$ at $t=0$, we can parameterize  $  \p \Omega_{ 0}  $ by some $\g_0 \in C^{1, \beta}(\TT)$. Then Theorem \ref{thm:flow_existence} shows that there are two flows $\g_i : \TT \times [0,T] \to \RR^2$ associated with each $\Omega_{i,t}$ with the same initial data $\g_0$.

Since the flows satisfy the regularity $ \g_i \in C  ( [0,T]; C^{\beta -2\alpha } ) $ and $\beta - 2\alpha > 2\alpha$, we can apply Theorem \ref{thm:uniqueness_by_flow} to obtain that
\begin{equation}
\delta(t) \lesssim \int_0^t \delta(\tau) \, d \tau  .
\end{equation} 
 It then follows from Gronwall that $\delta (t)= 0$ for $t \in [0,T]$.

Thus we must have $\g_1 (x,t)  =   \g_2 (x,t)$, and hence $  \Omega_{1,t} = \Omega_{2,t}$.

\subsection{Improved estimates for the \texorpdfstring{$W^{2,  p   }$}{W2p} case}

Using a similar construction to \cite[Lemma 3.7]{asqg}, one can show that there are examples of  $C^{1, \beta }$ domains $  \Omega$  where the velocity field $  v $ is not $C^{1, \beta }$ on $ \p \Omega$.

However, the below lemma shows that for $W^{2,  p   }$ patch solutions, the component $ \p_s v \cdot \T$ does reach the expected regularity $  C^{  1- \frac{1}{p}}$ using the maximal estimates from Lemma \ref{lemma:arc_length_estimates}.

For presentation simplicity,  we prove the $  C^{  1- \frac{1}{p}}$ regularity of $ \p_s v \cdot \T$ in the arc-length variable. The same regularity in terms of the Lagrangian label follows as well.

	\begin{lemma}\label{lemma:estimate_dvT_W2p}
		Let $0< \alpha< \frac{1}{2}$ and $ p > \frac{1}{1 - 2\alpha }   $. Suppose  $\Omega $ is a $W^{2,p}$ bounded domain and $v$ is the velocity field of $\Omega $ according to \eqref{eq:aSQG_BS}.
		
		Then the velocity field satisfies
		\begin{equation}\label{eq:estimate_v_W2p}
			\big|   \p_s v \cdot \T   \big|_{C^{0,\sigma }} \leq C(\alpha, p , \Omega)
		\end{equation}
		where $\sigma =1 - \frac{1}{p} >0$.
	\end{lemma}
	\begin{proof}
		We focus on the $C^{0,\sigma}$   continuity using the formula
		\begin{align*}
			\p_s v (s)  & =     P.V.  \int_{ \g    } \T (s'+s )   \frac{   (\g(s  )-\g( s'+s )) \cdot \T(s ) }{  |\g(s )-\g(s'+s )|^{2  +2\alpha }}  \,ds' .
		\end{align*}

	As in Lemma \ref{lemma:estimate_dv},	for any $\delta >0$, we denote by $\Delta_\delta f(s) =f(s+\delta) -f(s)  $   and consider the split
		\begin{align*}
			\Delta_\delta [ \p_s v   \cdot \T ]  & = (\p_s v  \cdot \T) (s  + \delta) - (\p_s v \cdot \T )  (s) = I_1 + I_2
		\end{align*}
		where
		\begin{align*}
			I_1 =  \lim_{\ep\to 0+} &\int_{ \ep< |  s' |    \leq 2\delta } \T ( s+  \delta+  s'  ) \cdot \T ( s+  \delta  ) \\
			& \qquad \times \frac{   (\g( s +\delta  )-\g(   s+  \delta+  s' )) \cdot \T( s +\delta ) }{  |\g( s+\delta )-\g(   s+  \delta+  s' )|^{2  +2\alpha }}    \,d   s' \\
			&\quad - \lim_{\ep\to 0+} \int_{ \ep< |  s' | \leq 2\delta   }\T (s'+s )  \cdot \T ( s  ) \frac{   (\g(s  )-\g( s'+s )) \cdot \T(s ) }{  |\g(s )-\g(s'+s )|^{2  +2\alpha }}  \,ds'
		\end{align*}
		and
		\begin{align*}
			I_2 =    \int_{   |  s'| \geq 2\delta   }  \Delta_\delta \left[ \T (s'+s ) \cdot \T ( s  )  \frac{   (\g(s  )-\g( s'+s )) \cdot \T(s ) }{  |\g(s )-\g(s'+s )|^{2  +2\alpha }}    \right] \,ds'.
		\end{align*}
		We will show the estimates $|I_1|,|I_2|  \lesssim \delta^{1- \frac{1}{p}}  $.
		
		\noindent
		\textbf{\underline{Estimate of $I_1$:}}

		By the estimates   \eqref{eq:arc_length_a}, \eqref{eq:arc_length_d}  and \eqref{eq:arc_length_e}   from Lemma \ref{lemma:arc_length_estimates}, for any $s,s' \in \RR$
		\begin{equation}\label{eq:aux_estimate_v_on_boundary_Cbeta}
			\begin{aligned}
				\T(s' ) \cdot \T (s) &= 1   + O(| s  - s' |^{2(1- \frac{1}{p}) } ) \\
				(\g(s      )-\g(  s' )) \cdot \T(s ) &= (s  - s' ) + O(|s  - s'|^{ 1+ 2(1- \frac{1}{p})  } ) \\
				\frac{|s  - s' |^{2  +2\alpha } }{|\g(s)-\g(  s')|^{2  +2\alpha }}   &  =1 + O(|s  - s'|^{     2(1- \frac{1}{p})   }  ) .
			\end{aligned}
		\end{equation}
		
		Using \eqref{eq:aux_estimate_v_on_boundary_Cbeta}, we have
		\begin{align*}
			I_1 = \lim_{\ep\to 0+} & \left( \int_{ \ep< |  s' |    \leq 2\delta }    \frac{ -   s'   }{  |  s'  |^{2  +2\alpha }} \,ds'  + \int_{ \ep< |  s' |    \leq 2\delta }  O(| s'|^{-1 - 2\alpha + 2(1- \frac{1}{p})  })  \,d   s' \right. \\
			&\quad -\left. \int_{ \ep< |  s' | \leq 2\delta   }    \frac{   - s'   }{  |  s'  |^{2  +2\alpha }}\,ds'   + \int_{ \ep< |  s' |    \leq 2\delta }  O(| s'|^{-1 - 2\alpha + 2(1- \frac{1}{p}) })  \,ds'\right).
		\end{align*}
		
		The first terms in the above two lines vanish by oddness, and we have
		\begin{align*}
			|I_1| & \lesssim     \int_{  | s'| \leq 2\delta   }   |    s' |^{    -1 - 2\alpha  + 2(1- \frac{1}{p})    }      \,ds' \lesssim \delta^{2-2\alpha - \frac{2 }{p} } \lesssim \delta^{ 1- \frac{1}{p}} 
		\end{align*}
		where the last inequality follows from $p> \frac{1}{1 -2\alpha }  $.

		\noindent
		\textbf{\underline{Estimate of $I_2$:}}

		In this regime, we need to consider the finite difference of the integrand
		\begin{align}\label{eq:aux_proofholder_Cbeta}
			\Delta_\delta \left[ \T (s'+s )\cdot \T(s )   \frac{   (\g(s  )-\g( s'+s )) \cdot \T(s ) }{  |\g(s )-\g(s'+s )|^{2  +2\alpha }}    \right] .
		\end{align}
		To reduce notation, for each fixed $|s'| > 0$, let us denote the function in \eqref{eq:aux_proofholder_Cbeta} by $K_{s'}(s)$, namely
		$$
		s \mapsto  K_{s'}(s):=  \T (s'+s ) \cdot \T(s )   \frac{   (\g(s  )-\g( s'+s )) \cdot \T(s ) }{  |\g(s )-\g(s'+s )|^{2  +2\alpha }} .
		$$
		From the regularity $\T, \N \in W^{1,p}$, we know that for each fixed $|s'| > 0$, $ s\mapsto K_{s'}$ is weakly differentiable with $L^p$ derivative:
		\begin{equation}
			\begin{aligned}
				\p_s K_{s'} (s) = & - \k(s'+s )\N (s'+s )\cdot \T(s )    \frac{   (\g(s  )-\g( s'+s )) \cdot \T(s ) }{  |\g(s )-\g(s'+s )|^{2  +2\alpha }}  \\
				& -\k( s )\T (s'+s )\cdot \N(s )    \frac{   (\g(s  )-\g( s'+s )) \cdot \T(s ) }{  |\g(s )-\g(s'+s )|^{2  +2\alpha }} \\
				& + \T (s'+s ) \cdot \T(s )     \frac{   (\T(s  )-\T( s'+s )) \cdot \T(s ) }{  |\g(s )-\g(s'+s )|^{2  +2\alpha }} \\
				& -\k(s) \T (s'+s ) \cdot \T(s )     \frac{   (\g(s  )-\g( s'+s )) \cdot \N(s ) }{  |\g(s )-\g(s'+s )|^{2  +2\alpha }}\\
				& -(2+2\alpha )\T (s'+s )  \cdot \T(s )    \frac{   (\g(s  )-\g( s'+s )) \cdot \T(s )     }{  |\g(s )-\g(s'+s )|^{4  +2\alpha }}\\
				& \qquad \times (\g(s  )-\g( s'+s )) \cdot (\T(s )  -\T( s'+s  ) ).
			\end{aligned}
		\end{equation}
		Using the   $W^{2,p}$ estimates in Lemma \ref{lemma:arc_length_estimates}, it is straightforward to show that $K_{s'}' $ satisfies the point-wise estimate
		\begin{equation}\label{eq:aux_K'_mean_value_0_Cbeta}
			|\p_s K_{s'} (s)| \lesssim       |s'|^{ -2\alpha  - \frac{1}{p}}  \left( \k(s+s')  + \k(s)+ \mathcal{M}\k(s) \right)  .
		\end{equation}
		So by the fundamental theorem of calculus for $W^{1,p} $ functions and the bound \eqref{eq:aux_K'_mean_value_0_Cbeta}, we have
		\begin{equation}\label{eq:aux_K'_mean_value_Cbeta}
			\begin{aligned}
				|K_{s'}(s+\delta) - K_{s'}(s ) |  & \leq \int_{s}^{s+\delta}  |\p_s K_{s'} (\tau )|\, d\tau \\
				& \lesssim   |s'|^{ -2\alpha  - \frac{1}{p}}  \delta^{1- \frac{1}{p}} .
			\end{aligned}
		\end{equation}

		Now we apply the estimate \eqref{eq:aux_K'_mean_value_Cbeta} for $K_{s'}(s)$ to $I_2$, obtaining
		\begin{align*}
			| I_2| \leq     \int_{ \frac{L}{2}  \geq |  s'| \geq 2\delta   }  \Delta_\delta \left[  K_{s'}(s) \right] \,ds' \lesssim \delta^{1- \frac{1}{p}}  \int_{    \frac{L}{2}  \geq |  s'| \geq 2\delta   }   |s'|^{ -2 \alpha - \frac{1}{p} }  \, ds' \lesssim \delta^{1- \frac{1}{p}}  
		\end{align*}
		where $ -2 \alpha - \frac{1}{p} < -1$ ensures the integral is bounded.

	\end{proof}

\subsection{Proof of \texorpdfstring{$W^{2,  p   }$}{W2p} uniqueness}

Given two  two $W^{2,p}$ patch solutions $\Omega_{i,t}$ on some $[0,T]$ that coincide at $t=0$, we need to show $  \Omega_{1,t} = \Omega_{2,t}$ for all time.

First, since $ C^{1,1 -\frac{1}{p}}  \subset W^{2,p}  $, Theorem \ref{thm:flow_existence} shows that there are two flows $\g_i : \TT \times [0,T] \to \RR^2$ associated with each $\Omega_{i,t}$ with the same initial data $\g_0$ and the regularity   $\g_i   \in C(  [0,T];C^{1 , 1 -\frac{1}{p} -2\alpha }(\TT)) $. Arguing similarly to the $C^{1,4\beta+}$ case, thanks to Theorem  \ref{thm:uniqueness_by_flow}, it suffices to show that this flow satisfies the improved regularity $\g_i   \in C(  [0,T];C^{1 , 1 -\frac{1}{p} }(\TT)) $.

Since the flow $\g_i $ parameterizes $\p  \Omega_{i,t} $ which is of class $W^{2,p}$, it suffices to show the metrics $g_i = |\p_x \g | \in  C(  [0,T];C^{1 -\frac{1}{p} }(\TT))$. By $C^{1 , 1 -\frac{1}{p} -2\alpha }$ regularity of $\g_i$ and Lemma \ref{lemma:estimate_dvT_W2p}, we have that $ \p_s v_i \cdot \T_i$ is also $C^{1 -\frac{1}{p} }$ in the corresponding parametrization $\g_i$. Viewing $\p_s v_i \cdot \T_i $ as   known $L^\infty (  [0,T];C^{1 -\frac{1}{p} }(\TT))$ functions, from \eqref{eq:metric_evolution} we obtain
\begin{equation}
	g_i( t) = |\p_x \g_i(0)| e^{\int_0^t \p_s v_i\cdot \T_i  } .
\end{equation}
This implies that  $g_i \in  C(  [0,T];C^{1 -\frac{1}{p} }(\TT))$.

With the regularity $ \g_i \in  C(  [0,T];C^{1 -\frac{1}{p} }(\TT))$,   we conclude by Theorem  \ref{thm:uniqueness_by_flow}  that $\g_1 (x,t)  =   \g_2 (x,t)$, and hence $  \Omega_{1,t} = \Omega_{2,t}$.

\section{Stability estimates}\label{sec:stablity}

We now derive suitable stability estimates of Theorem \ref{thm:uniqueness_by_flow} for two different patch solutions $  \Omega_{i,t} $ in terms of the  tangent vectors and arc-length metrics  of their boundaries $  \p \Omega_{i,t} $.

\subsection{Oddness of the kernel}

Since we will be frequently taking the difference between functions defined by the flows $\g_1 $ and $\g_2$, we introduce the notation $\Delta \left[f_i \right] := f_1 -f_2$ for any functions $f_i$ defined by $\g_i $ on $\TT$. For instance, $\Delta [\g_i](\xi) = \g_1(\xi) -\g_2(\xi)  $ and $\Delta \left[ \T_i(\xi) - \T_i(\eta) \right]= \left[ \T_1(\xi) - \T_1(\eta) \right]-\left[ \T_2(\xi) - \T_2(\eta) \right]    $. We will frequently use the telescoping formula
\begin{equation}\label{aux22122a}
\Delta [f_i g_i] = \Delta [f_i] g_1 +f_2 \Delta [ g_i].
\end{equation}
 
The stability estimates require the following result, related to the kernel appearing in the evolution of tangent vectors and arc-length metrics as in  \eqref{eq:estimate_v_W2p_Lagran}.

\begin{lemma}\label{lemma:kernel}
	Let $ 0 < \alpha < \frac{1}{2}$ and $ \beta >  2\alpha $. Assume that $\g_i : \TT \times [0,T] \to \RR^2$, $i=1,2$ are two $C^{1,\beta }$ flows of two corresponding $  C^{1,\beta }$ patch solutions $\Om_{i, t}$.
	
	 Let $\T_i$ and $g_i$ be the tangent and arc-length of each $\g_i$ and define $K_i: \TT \times \TT \to \RR  $
	\begin{equation}
		K_i(y,x ) = g_i(y) \frac{(\g_i(x) - \g_i(y) )\cdot \T_i(x)}{|\g_i(x) - \g_i(y)|^{2+2\alpha }} .
	\end{equation}
	
	Then for $|y-x|>0$,
	\begin{equation}\label{eq:lemma:kernel_1}
		| K_i(y,x ) + K_i(x,y )| \lesssim |x-y|^{-1 +\beta -2\alpha  }, 
	\end{equation}
	and for $|y-x|>0$ and any $z\in \RR $ such that $\max\{|z-x|, |z - y| \} \leq |x- y|$,
	\begin{equation}\label{eq:lemma:kernel_2}
		|\Delta [K_i(y,x )  ]  | \lesssim |x-y|^{-1 -2\alpha} \Big(    \mathcal{M} \Delta[g_i](z) + \mathcal{M} \Delta[\T_i](z) \Big)
	\end{equation}
	where $\mathcal{M} f$ denotes the periodic maximal function for any $f:\TT \to \RR $.

	Furthermore, for any $ a \in C^{\beta}(\TT)$, there holds
	\begin{equation}\label{eq:lemma:kernel_3}
		\Big| \int_{\TT} a(y) K_i(y,x) \, dy \Big| + \Big| \int_{\TT} a(x) K_i(y,x) \, dx  \Big| \lesssim |a|_{C^\beta }
	\end{equation}
	where the integrals are interpreted in the principal value sense.
\end{lemma}
\begin{proof}
	The assumption gives the $C^{\beta }$ regularity of  $ g_i$. Thus the first estimate follows directly from the $C^{\beta }$ regularity of $\T$ and   $ g$:
	\begin{equation}\label{eq:aux_kernel_1}
		\begin{aligned}
			\Big|  K_i(y,x ) +  K_i(x , y )  \Big| & \leq   \Big|  g_i (y) \frac{(\g_i(x) - \g_i(y) ) \cdot ( \T_i(x) - \T_i (y)}{|\g_i(x) - \g_i(y)|^{2+2\alpha }}   \Big|  \\
			& \qquad + \Big| (g_i (y) - g_i(x)) \frac{(\g_i(x) - \g_i(y) ) \cdot   \T_i(y)  }{|\g_i(x) - \g_i(y)|^{2+2\alpha }}   \Big| \\
			& \lesssim  \frac{ |x-y| |x-y|^\beta }{|x-y|^{2+2\alpha}}   .
		\end{aligned}
	\end{equation}
	Next, for \eqref{eq:lemma:kernel_2} we telescope to obtain
	\begin{equation}\label{eq:aux_kernel_2}
		\begin{aligned}
			\Delta [K_i(y,x )  ] & = \frac{\Delta [ (\g_i(x) - \g_i(y) ) ] \cdot \T_1(x)}{|\g_1(x) - \g_1(y)|^{2+2\alpha }}  \\
			& +  \frac{ (\g_2(x) - \g_2(y) )  \cdot \Delta [\T_i(x) ] }{|\g_1(x) - \g_1(y)|^{2+2\alpha }} \\
			& + (\g_2(x) - \g_2(y) )  \cdot  \T_2(x)  \Delta [ \frac{  1 }{|\g_i(x) - \g_i(y)|^{2+2\alpha }}].
		\end{aligned}
	\end{equation}
	
	We first prove that if $|x-y|>0$ and $z\in \RR $ such that $\max\{|z-x|, |z - y| \} \leq |x- y|$, then
	\begin{equation}\label{eq:aux_kernel_3}
		| \Delta [ (\g_i(x) - \g_i(y) ) ] | \lesssim |x-y| \mathcal{M}\Delta [ g_i] (z) + \mathcal{M}\Delta [   \T_i(z) ] .
	\end{equation}
	Indeed, by the fundamental theorem of calculus and the maximal function we have
	\begin{equation}\label{eq:aux_kernel_4}
		\begin{aligned}
			| \Delta [ (\g_i(x) - \g_i(y) ) ] | & \leq  \int_y^x | \Delta [ g_i( y' ) \T_i(y') ]| \, dy' \\
			& \leq \int^{z+ 2|x-y|}_{z-2|x-y|} | \Delta [ g_i( y' ) \T_i(y') ]| \, dy' \\
			& \lesssim |x-y| \mathcal{M}\Delta [ g_i] (z) + \mathcal{M}\Delta [   \T_i(z) ] .
		\end{aligned}
	\end{equation}
	
	For the first term in \eqref{eq:aux_kernel_2}, by \eqref{eq:aux_kernel_4} we have that
	\begin{align}
		| \Delta [ (\g_i(x) - \g_i(y) ) ] \cdot \T_1(x) | & \leq  \int_y^x | \Delta [ g_i(z) \T_i(z) ]| \, dz \nonumber \\
		& \lesssim  |x-y| \mathcal{M}\Delta [ g_i] (z) + \mathcal{M}\Delta [   \T_i(z) ] \label{eq:aux_kernel_5} .
	\end{align}

	For the second term in \eqref{eq:aux_kernel_2}, we can bound it by
	\begin{align}
		|  (\g_2(x) - \g_2(y) )  \cdot \Delta [ \T_i(x) ]| & \leq  \int_y^x | g_2(z)  \T_2(z)   \cdot \Delta [ \T_i(x) ]| \, dz  \nonumber \\
		& \lesssim  |x-y|    \Delta [ \T_i(x) ] . \label{eq:aux_kernel_6}
	\end{align}
	
	For the third term in \eqref{eq:aux_kernel_2}, notice that the mean value theorem for $ f(z) = |z|^{-2-2\alpha}$ applied to points $z_i = \g_i(x) - \g_i(y) $ and the fact that $|\g_i(x) - \g_i(y)|\sim |x-y| $ implies that
	\begin{align}\label{eq:aux_kernel_7}
		\Big| \Delta [ \frac{  1 }{|\g_i(x) - \g_i(y)|^{2+2\alpha }}] \Big|  \lesssim \frac{ \big|   \Delta [ |\g_i(x) - \g_i(y)|  ] \big|  }{ | x - y |^{3+2\alpha }   }.
	\end{align}
	
	It follows from \eqref{eq:aux_kernel_4} that
	\begin{align}\label{eq:aux_kernel_8}
		\Big| \Delta [ \frac{  1 }{|\g_i(x) - \g_i(y)|^{2+2\alpha }}] \Big|  \lesssim   |x-y|^{-2 - 2\alpha } \Big( \mathcal{M}\Delta [ g_i] (z) + \mathcal{M}\Delta [   \T_i(z) ]\Big).
	\end{align}
	By \eqref{eq:aux_kernel_8}, the third term in \eqref{eq:aux_kernel_2} satisfies the bound
	\begin{align}\label{eq:aux_kernel_9}
		&	\Big|  (\g_2(x) - \g_2(y) )  \cdot  \T_2(x)  \Delta [ \frac{  1 }{|\g_i(x) - \g_i(y)|^{2+2\alpha }}] \Big|   \nonumber \\
		& \lesssim |x-y|^{-1 -2\alpha} \Big(    \mathcal{M} \Delta[g_i](z) + \mathcal{M} \Delta[\T_i](z) \Big).
	\end{align}
	
Putting together \eqref{eq:aux_kernel_2}, \eqref{eq:aux_kernel_5}, \eqref{eq:aux_kernel_6}, and \eqref{eq:aux_kernel_9} we obtain the bound for $\Delta [K_i]$:
	\begin{equation*} 
		|\Delta [K_i(y,x )  ]  | \lesssim |x-y|^{-1 -2\alpha} \Big(    \mathcal{M} \Delta[g_i](z) + \mathcal{M} \Delta[\T_i](z) \Big).
	\end{equation*}

	Finally, we prove \eqref{eq:lemma:kernel_3}. It suffices to only consider $ \int a(y) K_i (y,x ) \, dy $. First, we split the integral 
	\begin{equation}\label{eq:aux_kernel_10}
		\begin{aligned}
 	&	\int a(y) K_i (y,x ) \, dy   = 	    \int_{\TT} a(y)  g_i(y) \frac{(\g_i(x) - \g_i(y) )\cdot \T_i(x)}{|\g_i(x) - \g_i(y)|^{2+2\alpha }}   \, dy \Big| \\
			& =    a(x)    g_i(x)  \int_{\TT}  \frac{(\g_i(x) - \g_i(y) )\cdot \T_i(x)}{|\g_i(x) - \g_i(y)|^{2+2\alpha }}   \, dy   \\
&  \qquad +   \int_{\TT} (a(y) -a(x))  g_i(y) \frac{(\g_i(x) - \g_i(y) )\cdot \T_i(x)}{|\g_i(x) - \g_i(y)|^{2+2\alpha }}   \, dy   \\
& \qquad \qquad  +    \int_{\TT}  a(x)  ( g_i(y) - g_i(y))  \frac{(\g_i(x) - \g_i(y) )\cdot \T_i(x)}{|\g_i(x) - \g_i(y)|^{2+2\alpha }}   \, dy  .
		\end{aligned}
	\end{equation}
Using $|a(x)  - a(y) | \leq  |a|_{C^\beta } |x-y |^\beta $ and $|g_i(x)  - g_i(y) | \lesssim |x-y |^\beta $, the last two integrals are absolutely convergent and  satisfy the desired bound, and so we focus on the first term in \eqref{eq:aux_kernel_10}.

We further split the integral
\begin{equation} \label{eq:aux_kernel_11}
\begin{aligned}
&  a(x)    g_i(x)  \int_{\TT}  \frac{(\g_i(x) - \g_i(y) )\cdot \T_i(x)}{|\g_i(x) - \g_i(y)|^{2+2\alpha }}   \, dy \\
&=  a(x)  g_i(x)  \int_{\TT} \frac{(\g_i(x) - \g_i(y) )\cdot    \T_i(y) }{|\g_i(x) - \g_i(y)|^{2+2\alpha }} \, dy   \\
&\qquad +  a(x)  g_i(x)  \int_{\TT} \frac{(\g_i(x) - \g_i(y) )\cdot (\T_i(y) - \T_i(x))}{|\g_i(x) - \g_i(y)|^{2+2\alpha }} \, dy  . 
\end{aligned}
\end{equation}
Observe that the first integral (in the principal value sense) on the right-hand side of \eqref{eq:aux_kernel_11} vanishes and the second integral is absolute convergent thanks to the $C^{1,\beta}$ regularity of $\g$	with $\beta > 2\alpha $. So   \eqref{eq:lemma:kernel_3} is proved.

\end{proof}

\subsection{Stability of the tangent evolution}
We are in the position to derive stability estimates the tangent evolution for $C^{1, 2 \alpha + }$ flows of the patch solutions.

Given two flows $\g_i$ of two patch solutions $\Omega_{i,t}$, we use the tangent equations
\begin{equation}\label{eq:aux_tangent_i}
\begin{cases}
	\p_t \T_i =  \p_s v_i \cdot \N_i \N_i &\\
\T_i |_{t = 0 } =  \frac{ \p_x \g_i }{|\g_i |} |_{t = 0 } .
\end{cases}
\end{equation}
 
Taking the difference of \eqref{eq:aux_tangent_i} for $i=1,2$, multiplying by $ \Delta [\T_i ] = \T_1 -\T_2$, and then integrating in space-time, we have
\begin{equation}
\begin{aligned}
	\frac{1}{2}|\T_1(t)  - \T_2 (t)|_{L^2}^2  - & \frac{1}{2}|\T_1(0)  - \T_2 (0)|_{L^2}^2   \\
& \leq 	  \int_{0}^t\int_{\TT} \Delta [ \p_s v_i \cdot \N_i ] \N_1  \cdot \Delta [  \T_i    ] (t)\,dx\,dt   \\
	& \qquad + \int_{0}^t\int_{\TT} \p_s v_2 \cdot \N_2  \Delta [ \N_i ]  \cdot \Delta [  \T_i   ] (t)  \,dx\,dt .
\end{aligned}
\end{equation}

This is the subject of the following proposition.

\begin{proposition}\label{prop:tangent_diffenrence}
Let $ 0 < \alpha < \frac{1}{2}$ and $ \beta  > 2\alpha $. Consider two $ C^{1,\beta }$ flows $\g_i : \TT \times [0,T] \to \RR^2$, $i=1,2$ of two   corresponding $  C^{1,\beta }$ patch solutions $\Om_{i, t}$. Let $\T_i$ and $g_i$ be the tangent and arc-length of each $\g_i$.

For any $t \in [0,T]$, there holds 
\begin{equation}\label{eq:tangent_diffenrence}
	\begin{aligned}
		|\T_1(t)  - \T_2 (t)|_{L^2}^2   \leq & |\T_1(0)  - \T_2 (0)|_{L^2}^2  \\
		&+  C \int_{0}^t|\T_1(\tau )  - \T_2 (\tau )|_{L^2}^2  +   |g_1(\tau  )  - g_2 ( \tau  )|_{L^2}^2  \,d \tau   .
	\end{aligned}
\end{equation}
\end{proposition}
\begin{proof}

By the boundedness  of $\p_s v$ due to Lemma \ref{lemma:estimate_dv}, it suffices to show
\begin{equation} \label{eq:aux_deltaT_1}
| I  | 
\lesssim  | \T_1   - \T_2  |_{L^2}^2 ,
\end{equation}
with
\begin{equation} \label{eq:aux_deltaT_2}
\begin{aligned}
I  & : = \int_{\TT} \Delta  [ \p_s v_i \cdot \N_i ] \N_1  \cdot   \Delta  [\T_i ] \,dx\\
& = \int_{\TT} \int_{\TT}  \Delta  [  \T_i (y)\N_i (x)   K_i(y,x ) ] \, \N_1(x)     \cdot  \Delta  [\T_i(x)] \,dy  \, dx,
\end{aligned}
\end{equation}
where we recall the kernels $K_i$
\begin{equation}\label{eq:aux_deltaT_4}
K_i(y,x ) =  g_i(y) \frac{(\g_i(x) - \g_i(y) )\cdot \T_i(x)}{|\g_i(x) - \g_i(y)|^{2+2\alpha }}.
\end{equation}

Telescoping \eqref{eq:aux_deltaT_2}, we have
\begin{equation}\label{eq:aux_deltaT_5}
\begin{aligned}
I & = I_1 +  I_2  ,
\end{aligned}
\end{equation}
where
\begin{equation}\label{eq:aux_deltaT_6}
\begin{aligned}
I_1 &=   \int_{\TT} \int_{\TT}  \Delta  [  \T_i (y)\cdot \N_i (x)]   K_1(y,x ) \, \N_1(x)      \cdot  \Delta  [\T_i(x)] \,dy  \, dx \\ 
I_2 &  =   \int_{\TT} \int_{\TT}    \T_2 (y) \cdot  \N_2 (x)   \Delta  [  K_i(y,x )] \, \N_1(x)    \cdot  \Delta  [\T_i(x)] \,dy  \, dx . 
\end{aligned}
\end{equation}

The task is then to estimate $I_1$ and $I_2$.

\noindent
\textbf{Estimate of $I_1$:}

We first reduce $I_1$ to estimating
\begin{equation}\label{eq:aux_deltaT_9}
I_1'   =   \int_{\TT} \int_{\TT}  \Delta  [  \T_i (y)\cdot \N_i (x)]  K_1(y,x ) \,   \underbrace{\N_1(y) }_\text{was $\N_1(x)$ in $I_1$}     \cdot  \Delta  [\T_i(x)] \,dy  \, dx
\end{equation}
since the commuting term has one extra factor $|\N_1(y) -\N_1(x) | \lesssim |x - y|^{\beta}$ making the kernel integrable.

A further telescoping gives the decomposition
\begin{equation}\label{eq:aux_deltaT_10}
I_1' = I_{11} + I_{12}
\end{equation}
with
\begin{equation}\label{eq:aux_deltaT_11}
I_{11} = \int_{\TT}\int_{\TT}  \Delta  [  \T_i (y)] \cdot \N_1 (x)  K_1(y,x ) \, \N_1(y)      \cdot  \Delta  [\T_i(x)] \,dy  \, dx ,
\end{equation}
and 
\begin{equation}\label{eq:aux_deltaT_12}
I_{12} = \int_{\TT} \int_{\TT} \T_2 (y)\cdot  \Delta  [   \N_i (x)]  K_1(y,x ) \, \N_1(y)      \cdot  \Delta  [\T_i(x)] \,dy  \, dx .
\end{equation}

For $I_{11}$, we first symmetrize the integral \eqref{eq:aux_deltaT_11} 
\begin{equation}\label{eq:aux_deltaT_13}
I_{11} = \frac{1}{2} \int_{\TT} \int_{\TT}  \Delta  [  \T_i (y)] \cdot \N_1 (x)   [  K_1(y,x ) +   K_1(x,y ) ]  \, \N_1(y)      \cdot  \Delta  [\T_i(x)] \,dy  \, dx,
\end{equation}
and then use \eqref{eq:lemma:kernel_1}  from Lemma \ref{lemma:kernel} to obtain
\begin{equation}
|I_{11}| \lesssim \big|\Delta  [\T_i ] \big|_{L^2}^2  .
\end{equation}

For $I_{12}$ we use H\"older 
\begin{equation}
|I_{12}| \leq \big|\Delta  [\T_i ] \big|_{L^2}^2 \sup_x \Big| \int_{\TT}    \T_2 (y)  \N_1(y)     K_1(y,x )    \,dy  \Big|  .
\end{equation}
and use \eqref{eq:lemma:kernel_3} from Lemma \ref{lemma:kernel} to conclude that it obeys the sought estimate due to the $C^\beta$ regularity of $\T,\N$.

\noindent
\textbf{Estimate of $I_2$:}

By \eqref{eq:lemma:kernel_2} from Lemma \ref{lemma:kernel} we have
\begin{equation}\label{eq:aux_deltaT_7}
|I_3| \lesssim   \int_{\TT} \int_{\TT}    |x-y|^{-1+\beta -2\alpha} \Big(  \mathcal{M} \Delta[g_i](y) + \mathcal{M} \Delta[\T_i](y)   \Big) \big|   \Delta  [\T_i(x)]\big|    \,dy  \, dx    .
\end{equation}
Thus by Young's inequality (thanks to  $\beta > 2\alpha $)and $L^2$-boundedness of the maximal function, we have
\begin{equation}\label{eq:aux_deltaT_8}
|I_3| \lesssim  \Big[  \big| \Delta[g_i] \big |_{L^2}  +  \big| \Delta[\T_i] \big|_{L^2}   \Big]  \big|\Delta  [\T_i ] \big|_{L^2} 
\end{equation}
which   satisfies the claimed bound.

\end{proof}

\subsection{Stability of the arc-length evolution}

Next, we show the same stability estimate for the arc-length difference and trajectory difference.
\begin{proposition}\label{prop:metric_diffenrence}
Under the setting of Proposition \ref{prop:tangent_diffenrence}, there holds
\begin{equation}\label{eq:metric_diffenrence}
	\begin{aligned}
		| g_1(t)  - g_2 (t)|_{L^2}^2  \leq &| g_1(0)  -  g_2 (0)|_{L^2}^2 \\
		&  + C \int_{0}^t|\T_1( \tau )  - \T_2 ( \tau )|_{L^2}^2  +   |g_1( \tau )  - g_2 ( \tau )|_{L^2}^2  \,d \tau ,
	\end{aligned}
\end{equation}
and
\begin{equation}\label{eq:trajectory_diffenrence}
	\begin{aligned}
		| \g_1(t)  - \g_2 (t)|_{L^2}^2  \leq &| \g_1(0)  - \g_2 (0)|_{L^2}^2 \\
		&  + C \int_{0}^t|\T_1(\tau)  - \T_2 (\tau)|_{L^2}^2  +   |g_1(\tau)  - g_2 (\tau)|_{L^2}^2  \,d \tau .
	\end{aligned}
\end{equation}
\end{proposition}
\begin{proof}

We first prove the bound \eqref{eq:metric_diffenrence} for $\g_i$.
By the integral equation  of the flows  $\g_i $,
\begin{align*}
| \g_1 (t) -\g_2  (t) |_{L^2} & \leq  | \g_1(0)  - \g_2 (0)|_{L^2}^2 \\
&\qquad + \int_0^t \Big| \int_{\TT} \Delta \Big[ \T_i (y) g_i(y) \frac{1}{  |\g_i (x) - \g_i(y)|^{2\alpha } } \Big] \,dy \Big|_{L^2}.
\end{align*}

Once again, dropping the time variable we need to prove
\begin{equation}\label{eq:aux_traj_diff_1}
 \Big| \int_{\TT} \Delta \Big[ \T_i (y) g_i(y) \frac{1}{  |\g_i (x) - \g_i(y)|^{2\alpha } } \Big] \,dy \Big|_{L^2}  \lesssim   |\T_1   - \T_2  |_{L^2}^2  +   |g_1  - g_2 |_{L^2}^2 .
\end{equation}

Since $2\alpha<1$, the kernel $\frac{1}{  |\g_i (x) - \g_i(y)|^{2\alpha } }$ is integrable, and so by telescope \eqref{eq:aux_traj_diff_1} it suffices to prove 
$$
 \Big| \int_{\TT}  \T_2 (y) g_2(y) \Delta \Big[ \frac{1}{  |\g_i (x) - \g_i(y)|^{2\alpha } } \Big] \,dy \Big|_{L^2} \lesssim   |\T_1   - \T_2  |_{L^2}^2  +   |g_1  - g_2 |_{L^2}^2 .
$$
This estimate relies on the following bound
\begin{equation}
		\Big| \Delta [ \frac{  1 }{|\g_i(x) - \g_i(y)|^{ 2\alpha }}] \Big|  \lesssim   |x-y|^{ - 2\alpha } \Big( \mathcal{M}\Delta [ g_i] (x) + \mathcal{M}\Delta [   \T_i(x) ]\Big),
\end{equation}
whose proof follows from a small modification of the argument used for \eqref{eq:aux_kernel_8}

Next, for  $g_1 -g_2 $, similarly to $\T_1 -\T_2$  we  will bound the following inequality
\begin{equation*}
\begin{aligned}
	\frac{1}{2}| g _1(t)  - g_2 (t)|_{L^2}^2 & \leq \frac{1}{2}| g _1(0)  - g_2 (0)|_{L^2}^2  \\
& \qquad +  \int_{0}^t\int_{\TT} \Delta [   \p_s v_i \cdot \T_i  g_i ]    ( g_1(t)  - g_2 (t))  \,dx\,dt  .  
\end{aligned}
\end{equation*}
Again, dropping the time variable   it suffices to show 
\begin{equation}\label{eq:aux_trajectory_diffenrence_3}
\Big|  \int_{\TT} \Delta [   \p_s v_i \cdot \T_i  ]  g_1    ( g_1   - g_2 )  \,dx  \Big| \lesssim  |\T_1   - \T_2  |_{L^2}^2  +   |g_1  - g_2 |_{L^2}^2 .
\end{equation}
Telescoping the factor $\Delta [   \p_s v_i \cdot \T_i  ]$ and appealing the the definition of $K_i$, we have
\begin{equation}
\Big|  \int_{\TT} \Delta [   \p_s v_i \cdot \T_i     ]    g_1 ( g_1   - g_2 )  \,dx  \Big| = J_1 + J_2+ J_3
\end{equation}
where
\begin{align}
J_1 & =  \int_{\TT}\int_{\TT}  \Delta  [  \T_i (y)] \cdot \T_1 (x)   K_1(y,x )  g_1(x)    \Delta  [g_i(x)] \,dy  \, dx\\
J_2  & =  \int_{\TT}\int_{\TT}    \T_2 (y) \cdot \Delta  [ \T_i (x)   ] K_1(y,x )  g_1(x)    \Delta  [g_i(x)] \,dy  \, dx\\
J_3 & = \int_{\TT} \int_{\TT} \T_2 (y) \cdot \T_2 (x)   \Delta  [   K_i(y,x )     ] g_1(x)    \Delta  [g_i(x)] \,dy  \, dx .
\end{align}

\noindent
\textbf{Estimate of $J_1$:}

Similar to the case of $I_1$ in the Proposition \ref{prop:tangent_diffenrence},  we can first reduce $J_1$ to
\begin{align*}
J_1'   & = \int_{\TT} \int_{\TT}  \Delta  [  \T_i (y)] \cdot \T_1 (y)   K_1(y,x )  g_1(x)    \Delta  [g_i(x)] \,dy  \, dx
\end{align*}
since the error term has an extra factor $ \T_1 (x)  -\T_1 (y) $.

By the identity $    \Delta  [  \T_i (y)]  \cdot \T_1 (y) =    \frac{1}{2}  ( \T_1 (y) -\T_2 (y))^2    $,  
\begin{align*}
J_1' & =   \frac{1}{2} \int_{\TT} \int_{\TT}   ( \T_1 (y) -\T_2 (y))^2     K_1(y,x )  g_1(x)    \Delta  [g_i(x)] \,dy  \, dx. 
\end{align*}

Then we use \eqref{eq:lemma:kernel_3} from Lemma \ref{lemma:kernel} for the $x$ variable together with H\"older's inequality in $y$ to obtain
\begin{equation}
\begin{aligned}\label{eq:aux_trajectory_diffenrence_4}
 \big| J_1'\big|  & \lesssim | \Delta  [  \T_i  ]    |_{L^2}^2 \sup_i   \big| \int_{\TT}      K_1(y,x )      \Delta  [g_i(x)]    \, dx  \\
& \lesssim   \big| \Delta  [  \T_i  ]    \big| |_{L^2}^2. 
\end{aligned}
\end{equation}

\noindent
\textbf{Estimate of $J_2$:}

For this integral, the estimates follows directly from H\"older's inequality in the $x$ variable and \eqref{eq:lemma:kernel_3} from Lemma \ref{lemma:kernel}:
\begin{equation}\label{eq:aux_trajectory_diffenrence_5}
\begin{aligned}
 |J_2| & \lesssim |\Delta [\T_i ] |_{L^2} |\Delta [g_i ]|_{L^2}  \sup_{x} \Big|   \int   \T_2 (y)  K_1(y,x )  g_1(y)      \,dy  \Big|\\
& \lesssim |\Delta [\T_i ] |_{L^2}^2 +  |\Delta [g_i ]|_{L^2}^2  .
\end{aligned}
\end{equation}

\noindent
\textbf{Estimate of $J_3$:}

By \eqref{eq:lemma:kernel_2} from Lemma \ref{lemma:kernel}, we have that 
\begin{equation}\label{eq:aux_trajectory_diffenrence_6}
\begin{aligned}
| J_3 | & \lesssim\int_{\TT} \int_{\TT}   |x-y|^{\beta } \Big|  \Delta  [   K_i(y,x )     ]     \Big|  \Big| \Delta  [g_i(x)] \Big| \,dy  \, dx\\
&  \lesssim \int_{\TT}\int_{\TT}  |x-y|^{-1 +\beta -2\alpha} \Big(    \mathcal{M} \Delta[g_i](y) + \mathcal{M} \Delta[\T_i](y) \Big) \Big| \Delta  [g_i(x)] \Big| \,dy  \, dx \\
&\lesssim \Big(  \big|  \mathcal{M} \Delta [g_i]\big|_{L^2}   + \big|\mathcal{M} \Delta [\T_i]    \big|_{L^2}    \Big)|   \Delta [g_i]    |_{L^2} .
\end{aligned}
\end{equation}

All three estimates, \eqref{eq:aux_trajectory_diffenrence_4}, \eqref{eq:aux_trajectory_diffenrence_5}, and \eqref{eq:aux_trajectory_diffenrence_6} of $J_i$ are compatible with the bound \eqref{eq:aux_trajectory_diffenrence_3}, and hence the proposition is proved.

\end{proof}

\subsection{Proof of Theorem \ref{thm:uniqueness_by_flow}}

By the assumption on the flows, we have $  \g_i \in L^\infty(0,T; C^{1,\beta }(\TT))$. It follows that the right-hand side of the equations for the flows $\g_i$, the arc lengths $g_i$ \eqref{eq:metric_evolution}  and tangent $\T_i$ \eqref{eq:tangent_evolution} are all $L^2(\TT)$. Therefore $t \mapsto \delta(t)$ is continuous (in fact Lipschitz).

The stability estimate of Theorem \ref{thm:uniqueness_by_flow} follows immediately  once we combine Proposition \ref{prop:tangent_diffenrence} and Proposition \ref{prop:metric_diffenrence}

\bibliographystyle{alpha}
\bibliography{asqg_patch}

\end{document}